\documentclass{compositio}
\renewcommand{\contentsname}{Contents\\{\footnotesize\normalfont(A table
of contents should normally not be included)}}
\usepackage[english]{babel}
\usepackage{newlfont}
\usepackage{amssymb,amsfonts}
\usepackage{mathtools,braket,mathrsfs}
\usepackage[all]{xy}

\usepackage[colorinlistoftodos,bordercolor=black,
backgroundcolor=orange!70!white,linecolor=black,]{todonotes}

\DeclareFontFamily{U}{wncy}{}
    \DeclareFontShape{U}{wncy}{m}{n}{<->wncyr10}{}
    \DeclareSymbolFont{mcy}{U}{wncy}{m}{n}
    \DeclareMathSymbol{\Sh}{\mathord}{mcy}{"58}

\newcommand{\bb}{\mathbb}
\newcommand{\mbf}{\mathbf}

\newcommand{\scr}{\mathscr}
\newcommand{\mrm}{\mathrm}

\theoremstyle{definition}
\newtheorem{theorem}{Theorem}[section]
\newtheorem{lemma}[theorem]{Lemma}
\newtheorem{proposition}[theorem]{Proposition}
\newtheorem{corollary}[theorem]{Corollary}
\newtheorem{definition}[theorem]{Definition}

\theoremstyle{remark}
\newtheorem{remark}{Remark}
\newtheorem{example}[theorem]{Example}

\begin{document}

\title{Growth of the analytic rank of modular elliptic curves over quintic extensions}

\author{Michele Fornea}
\email{michele.fornea@mail.mcgill.ca}
\address{McGill University, Montreal, Canada.}
\date{February 10, 2018}
%\dedication{A dedication can be included here.}
%least one subject code is required. Please refer to
%\url{http://www.ams.org/msc/} for a list of codes.}
%\thanks{This file documents \pkg{compositio} version \Fileversion\ and
%was last revised \Filedate.}

\begin{abstract}
Given $F$ a totally real field and $E_{/F}$ a modular elliptic curve, we denote by $G_5(E_{/F};X)$ the number of quintic extensions $K$ of $F$ such that the norm of the relative discriminant is at most $X$ and the analytic rank of $E$ grows over $K$, i.e., $r_\mrm{an}(E/K)>r_\mrm{an}(E/F)$. We show that $G_5(E_{/F};X)\asymp_{+\infty} X$ when the elliptic curve $E_{/F}$ has odd conductor  and at least one prime of multiplicative reduction. As Bhargava, Shankar and Wang \cite{BSW} showed that the number of quintic extensions of $F$ with norm of the relative discriminant at most $X$ is asymptotic to $c_{5,F} X$ for some positive constant $c_{5,F}$, our result exposes the growth of the analytic rank as a very common circumstance over quintic extensions.
\end{abstract}

\maketitle

\vspace*{6pt}\tableofcontents  % for this guide only.
% A table of contents should normally not be included

%%%%%%%%%%%%%%%%%%%%%%%%%%%%%%%%%%%%%%%%%%%%%%%%%%%%%%
% Introduction
%%%%%%%%%%%%%%%%%%%%%%%%%%%%%%%%%%%%%%%%%%%%%%%%%%%%%%
\section{Introduction}
The arithmetic of elliptic curves is an intriguing mystery for number theorists. Given an elliptic curve $E$ over a number field $K$, it is possible to package its arithmetic information into a generating series $L(E/K,s)$. While a priori the series just converges for $\mrm{Re}(s)\gg0$, conjecturally it has analytic continuation to the whole complex plane and a functional equation $s\mapsto 2-s$ with center $s=1$. The analytic rank is defined as the conjectural order of vanishing at the center $r_\mrm{an}(E/K)=\mrm{ord}_{s=1}L(E/K,s)$. This analytic invariant has an algebraic doppelg{\" a}nger: the $\bb{Z}$-rank of the finitely generated abelian group of $K$-rational points, $r_\mrm{alg}(E/K)$, which is called the algebraic rank of $E_{/K}$. 
The BSD-conjecture, inspired by the pioneering work of Birch and Swinnerton-Dyer, claims the equality of the two invariants. Therefore, when the elliptic curve $E$ is defined over a totally real field $F$ and $K/F$ is a finite extension, we expect the inequality $r_\mrm{an}(E/K)\ge r_\mrm{an}(E/F)$ to hold. Furthermore, the strict inequality should be explained by the presence of a non-torsion point in $E(K)$ linearly independent from $E(F)$. We like to think about our main result as evidence for the fact that there should be a systematic way to produce non-torsion points over $S_5$-quintic extensions of totally real fields, in analogy with the case of Heegner points over CM fields.

Our main result is compatible with the conjectures in \cite{CFK}, \cite{CFK2} about the growth of the analytic rank of rational elliptic curves over cyclic quintic extensions. In those works growth is predicted to be a rare phenomenon, however, cyclic quintic extensions form a thin subset of all quintic extensions: the counting function of cyclic quintic fields is asymptotic to $\alpha X^{1/4}$ for some positive constant $\alpha>0$ \cite{WAbelian}. Finally, we would like to remark that all elliptic curves over a totally real field $F$ with $[F:\bb{Q}]\le2$ are modular and that, in general, all but finitely many $\overline{\bb{Q}}$-isomorphism classes of elliptic curves over a totally real field $F$ are known to be modular (\cite{W}, \cite{TW}, \cite{BCDT}, \cite{EllipticReal}) making our result widely applicable.

\subsubsection{Strategy of the proof.}
Let $F$ be a totally real field, $K/F$ an $S_5$-quintic extension with a totally complex Galois closure $J$ such that the subfield of $J$ fixed by $A_5$ is a totally real quadratic extension $M/F$.
For $E_{/F}$ a modular elliptic curve corresponding to a primitive Hilbert cuspform $f_E$ of parallel weight two, the key idea of the paper is to interpret the ratio of $L$-functions $L(E/K, s)/L(E/F, s)$ as the twisted triple product $L$-function attached to $f_E$ and a certain Hilbert cuspform $g$ over $M$ of parallel weight one. Then, the sign $\varepsilon_{K/F}$ of the functional equation of $L(E/K, s)/L(E/F, s)$ is determined by the splitting behaviour in $K$ of the primes of multiplicative reduction of $E_{/F}$, and we can show there is a positive proportion of quintic extensions $K/F$ for which $\varepsilon_{K/F} = -1$ by invoking  \cite{BSW}.

The twisted triple product $L$-function, attached to a modular elliptic curve $E_{/F}$ and a cuspform $g$ of parallel weight one over a totally real quadratic extension $M/F$, is the $L$-function $L(E,\otimes\mbox{-}\mrm{Ind}_M^F(\varrho_g),s)$. Here, $\otimes\mbox{-}\mrm{Ind}_M^F(\varrho_g)$ denotes the tensor induction of the Artin representation attached to $g$. The main technical result of our work consists in proving the existence of an eigenform $g$ such that $\otimes\mbox{-}\mrm{Ind}_M^F(\varrho_g)=\mrm{Ind}_K^F\bb{I}-\bb{I}$. Thanks to the modularity of totally odd Artin representations \cite{PS}, the problem reduces to finding the solution of a Galois embedding problem as follows.
The group $G(J/M)\cong A_5$ does not afford any irreducible $2$-dimensional representation, but it has two conjugacy classes of embeddings into $\mrm{PGL}_2(\bb{C})$. Therefore, we look for a lift of the $2$-dimensional projective representation of $G_M\twoheadrightarrow G(J/M)\hookrightarrow \mrm{PGL}_2(\bb{C})$ which $(i)$ is totally odd, $(ii)$ has controlled ramification, and $(iii)$ whose tensor induction is $\mrm{Ind}_K^F\bb{I}-\bb{I}$.
Note that every projective 2-dimensional representation has a minimal lift with index a power of $2$ (Lemma 1.1, \cite{Quer}), thus we are led to consider the following Galois embedding problem:
\begin{quote}
 Given a finite set of primes $\Sigma_0$, is it possible to find a Galois extension $H/F$ unramified at $\Sigma_0$, containing $J/F$ and such that, 
$1\to\scr{C}_{2^r}\to G(H/F)\to G(J/F)\to 1$ is a non-split extension for some $r\ge1$?
\end{quote}
Here $\scr{C}_{2^r}$ denotes the cyclic group of order $2^r$  considered as an $S_5$-module via the homomorphism $
S_5\twoheadrightarrow \{\pm1\}\hookrightarrow\text{Aut}(\scr{C}_{2^r})$,
taking the non-trivial element of $\{\pm1\}$ to the automorphism $x\mapsto x^{-1}$. In Theorem $\ref{ginger}$, we are able to provide conditions for the Galois embedding problem to have a solution.

%%%%%%%%%%%%%%%%%%%%%%%%%%%%%%%%%%%%%%%%%%%%%%%%%%%%%%
% ACKNOWLEDGEMENTS
%%%%%%%%%%%%%%%%%%%%%%%%%%%%%%%%%%%%%%%%%%%%%%%%%%%%%%

\begin{acknowledgements}
This work would not have been possible without the constant guidance and support of my Ph.D. advisors Henri Darmon and Adrian Iovita. 	
I would like to thank Jan Vonk for pointing out the relevance of the author's work for statistical questions about elliptic curves, and I am grateful to Chantal David and Hershy Kisilevsky for useful comments on a first draft of the paper. Finally, I would like to thank the anonymous referee for the valuable feedback on the paper.

\end{acknowledgements}

%%%%%%%%%%%%%%%%%%%%%%%%%%%%%%%%%%%%%%%%%%%%%%%%%%%%%%
% On exotic tensor inductions
%%%%%%%%%%%%%%%%%%%%%%%%%%%%%%%%%%%%%%%%%%%%%%%%%%%%%%
\section{On exotic tensor inductions}

Let $A$, $B$ be groups, $n\in\mathbb{N}$ and $\phi:A\to S_n$ a group homomorphism. The wreath product of $B$ with $A$ is $B\wr A= B^{\oplus n}\rtimes_\phi A$, where $A$ acts permuting the factors through $\phi$. 

Let $G$ be a group and $Q$ a subgroup of index $n$. Denote by $\pi:G\to S_n$ the action of $G$ on right cosets by right multiplication and let $\{g_1,\dots, g_n\}$ be a set of coset representatives. For any $g\in G$ and $i\in\{1,\dots, n\}$, we denote by $q_i(g)$ the unique element of $Q$ such that $g\cdot g_i=g_{i\pi(g)}\cdot q_i(g)$.
The map $\varphi: G\to Q\wr S_n$, given by $g\mapsto (q_1(g),\dots, q_n(g), \pi(g))$, is an injective group homomorphism. Moreover, a different choice of coset representatives produces a homomorphism conjugated to $\varphi$ by an element of $G$. 

\begin{definition}
Let $Q$ be a subgroup of $G$ of index $n$, $\varrho:Q\to$ $\mrm{Aut}(V)$ a representation of $Q$. We define the tensor induction $\otimes\mbox{-}\mrm{Ind}_Q^G(\varrho)$ as the composition of the arrows in the following diagram
\[\xymatrix{
G\ar[d]_\varphi\ar@{.>}[drrrr]^{\otimes\mbox{-}\mrm{Ind}_Q^G(\varrho)} & & & &\\
Q\wr S_n\ar[rr]_{(\varrho,\mrm{id}_{S_n})}&& \text{Aut}(V)\wr S_n\ar[rr]_{(\alpha,\psi)}&& \text{Aut}(V^{\otimes n}),
}\]
where $\alpha:\text{Aut}(V)^{\oplus n}\to \text{Aut}(V^{\otimes n})$ is given by
$\alpha(f_1,\dots,f_n)= f_1\otimes\dots\otimes f_n$, and $\psi:S_n\to \hbox{Aut}(V^{\otimes n})$ by $\sigma\mapsto [\psi(\sigma): v_1\otimes\dots\otimes v_n\mapsto v_{1\sigma}\otimes\dots\otimes v_{n\sigma}]$.
\end{definition}

\begin{example}
Suppose $Q$ is a subgroup of of $G$ index $2$ and let $\{1,\theta\}$ be representatives for the right cosets, then
\[\begin{array}{lll}
q_1(g)=g, & q_2(g)=\theta g\theta^{-1} & \text{if}\ g\in Q\\
q_1(g)=g\theta^{-1}, &   q_2(g)=\theta g & \text{if}\ g\in G\setminus Q.
\end{array}\]
Thus, 
\[
\otimes\mbox{-}\mrm{Ind}_Q^G(\varrho)(g)=
\begin{cases}
\varrho(g)\otimes\rho(\theta g\theta^{-1})& g\in Q\\
[\varrho(g\theta^{-1})\otimes\varrho(\theta g)]\circ\psi(12)& g\in G\setminus Q.
\end{cases}
\]
\end{example}

\begin{proposition}\label{tre}
Let $Q$ be a subgroup of index $2$ of $G$ and $\{1,\theta\}$ be representatives for the right cosets. If $(V,\varrho)$ is an irreducible $2$-dimensional representation of $Q$ with projective image isomorphic to either $A_4, S_4$ or $A_5$, then the tensor induction $(V^{\otimes G},\otimes\mbox{-}\mrm{Ind}_Q^G(\varrho))$ is reducible if and only if $V^*(\lambda)\cong V^\theta$ for some character $\lambda:Q\to\bb{C}^\times$. When that happens the decomposition type is $(3,1)$.
\end{proposition}
\begin{proof}
If $V^*(\lambda)\cong V^\theta$ then the tensor product factors as $V\otimes V^\theta\cong\text{Ad}^0(V)(\lambda)\oplus\bb{C}(\lambda)$,
where $\text{Ad}^0(V)$ is irreducible (Lemma 2.1, \cite{DLR}). By Frobenius reciprocity,
$\mrm{Hom}_G(V^{\otimes G},\mrm{Ind}_Q^G(\lambda))=\mrm{Hom}_Q(V\otimes V^\theta, \bb{C}(\lambda))\not=0$, hence $V^{\otimes G}$ is reducible. Since $\big(V^{\otimes G}\big)_{\lvert Q}=V\otimes V^\theta$ has decomposition type $(3,1)$, so does $V^{\otimes G}$.

Suppose $V^{\otimes G}$ is reducible. If $V^{\otimes G}$ contains a $1$-dimensional subrepresentation then $V^*(\lambda)\cong V^\theta$. Indeed, if $\bb{C}(\chi)$ is a subrepresentation of $V^{\otimes G}$ then
$0\not=\mrm{Hom}_G(V^{\otimes G}, \bb{C}(\chi))\hookrightarrow \mrm{Hom}_Q(V\otimes V^\theta, \bb{C}(\chi_{\lvert Q}))$. Therefore, the tensor product $V\otimes V^\theta(\chi^{-1}_{\lvert Q})$ has a non-zero $Q$-invariant vector, i.e. $0\not=\mrm{H}^0(Q,V\otimes V^\theta(\chi^{-1}_{\lvert Q}))=\mrm{Hom}_Q(V^*(\chi_{\lvert Q}), V^\theta)$, which implies that  $V^*( \chi_{\lvert Q})\cong V^\theta$ given the irreducibility of $V$. Then we can apply the previous step to conclude. Finally, if $V^{\otimes G}$ has decomposition type $(2,2)$, at least one of the irreducible component decomposes into a sum of characters  when restricted to $Q$ (Lemma 2.2, \cite{DLR}), but then (Lemma 2.1, \cite{DLR}) produces a contraddiction.
\end{proof}

%%%%%%%%%%%%%%%%%%%%%%%%%%%%%%%%%%%%%%%%%%%%%%%%%%%%%%
% Galois embedding problems
%%%%%%%%%%%%%%%%%%%%%%%%%%%%%%%%%%%%%%%%%%%%%%%%%%%%%%

\section{Galois embedding problems}
\subsubsection{Cohomological computation.}
Let $F$ be a totally real number field, $\Sigma_0$ a finite set of places of $F$ disjoint from the set $\Sigma_\infty$ of archimedean places and the set $\Sigma_2$ of places above $2$. For $\Sigma$ the complement of $\Sigma_0$, we let $G_{F,\Sigma}$ denote the Galois group of the maximal Galois extension $F_\Sigma$ of $F$ unramified outside $\Sigma$. We consider $M/F$ a totally real quadratic extension unramified outside $\Sigma$, and for all $r\ge1$ we give $\scr{C}_{2^r}$ the structure of $G_{F,\Sigma}$-module via the homomorphism 
$G_{F,\Sigma}\twoheadrightarrow G(M/F)\hookrightarrow\text{Aut}(\scr{C}_{2^r})$ 
taking the non-trivial element of $G(M/F)$ to the automorphism $x\mapsto x^{-1}$. We denote by $\scr{M}_2=\lim_{r} \scr{C}_{2^r}$ the $G_{F,\Sigma}$-module obtained by taking the direct limit with respect to the natural inclusions $\scr{C}_{2^r}\to \scr{C}_{2^{r+1}}$.
Let $\scr{C}_{2^r}'$ be the dual Galois module $\mrm{Hom}_{\mrm{Gr}}\big(\scr{C}_{2^r}, \cal{O}_\Sigma^\times\big)$, where $\cal{O}_\Sigma$ is the ring of $\Sigma$-integers in $F_\Sigma$. As a $G_M$-module $\scr{C}_{2^r}'$ is isomorphic to the group $\mu_{2^r}$ of $2^r$th-roots of unity with the natural Galois action, hence the field $M_r=M(\mu_{2^r})$ trivializes $\scr{C}_{2^r}'$.

We are interested in analyzing the maps between the various kernels
\[
\Sh^1(G_{F,\Sigma},\scr{C}_{2^r}')=\ker\left(\mrm{H}^1(G_{F,\Sigma}, \scr{C}_{2^r}')\longrightarrow\prod_{v\in\Sigma} \mrm{H}^1(F_v, \scr{C}_{2^r}') \right).
\]

\begin{proposition}\label{pio}
For all $r\ge2$ the map $(j_{r}')_*: \Sh^1(G_{F,\Sigma}, \scr{C}_{2^{r}}')\to \Sh^1(G_{F,\Sigma}, \scr{C}_{2^{r-2}}')$, induced by the dual of the natural inclusion $j_r: \scr{C}_{2^{r-2}}\to \scr{C}_{2^{r}}$, is zero.
\end{proposition}
\begin{proof}
We claim that the restriction $\mrm{H}^1(G_{M_r,\Sigma}, \scr{C}_{2^r}')\to \prod_{w\in\Sigma(M_r)}\mrm{H}^1(M_{r,w}, \scr{C}_{2^r}')$ is injective, where the product is taken over all places of $M_r$ above a place in $\Sigma$. Indeed, if $\phi:G_{M_r}\to \scr{C}_{2^r}'$ is in the kernel of the restriction map, then the field fixed by $\ker\phi$ is a Galois extension of $M_r$ in which the primes that split completely have density $1$. Cebotarev's density theorem implies that such extension is $M_r$ itself. By examining the commutative diagram 
\[\xymatrix{
& & \mrm{H}^1(G_{M_r,\Sigma}, \scr{C}_{2^r}')\ar @{^{(}->}[r]& \underset{w\in \Sigma(M_r)}{\prod}\mrm{H}^1(M_{r,w}, \scr{C}_{2^r}')\\
0\ar[r]& \Sh^1(G_{F,\Sigma},\scr{C}_{2^r}')\ar[r]& \mrm{H}^1(G_{F,\Sigma}, \scr{C}_{2^r}')\ar[r]\ar[u]&  \underset{v\in\Sigma}{\prod}\mrm{H}^1(F_v, \scr{C}_{2^r}')\ar[u]&\\
0\ar[r]& \Sh^1(M_r/F,\scr{C}_{2^r}')\ar[r]\ar@{^{(}->}[u]& \mrm{H}^1(M_r/F, \scr{C}_{2^r}')\ar[r]\ar@{^{(}->}[u]&  \underset{v\in\Sigma}{\prod}\mrm{H}^1(M_{r,w}/F_v, \scr{C}_{2^r}'),\ar@{^{(}->}[u]&
}\] we see that  $\Sh^1(G_{F,\Sigma},\scr{C}_{2^r}')=\Sh^1(M_r/F,\scr{C}_{2^r}')$.

We claim that $\Sh^1(G_{F,\Sigma},\scr{C}_{2^r}')$ is killed by multiplication by $4$. Clearly, it suffices to prove that $\mrm{H}^1(M_r/F, \scr{C}_{2^r}')$ is killed by multiplication by $4$. Considering the inflation-restriction exact sequence
\[\xymatrix{
0\ar[r]&\mrm{H}^1(M/F, (\scr{C}_{2^r}')^{G(M_r/M)})\ar[r]& \mrm{H}^1(M_r/F, \scr{C}_{2^r}')\ar[r]&\mrm{H}^1(M_r/M, \scr{C}_{2^r}').
}
\]
and the fact that both $\mrm{H}^1(M/F, (\scr{C}_{2^r}')^{G(M_r/M)})$ and $\mrm{H}^1(M_r/M, \scr{C}_{2^r}')$ are isomorphic to $\bb{Z}/2\bb{Z}$ (Lemma 9.1.4 and Proposition 9.1.6, \cite{Neu}), the claim follows. 

There is a natural factorization of multiplication by $4$ on $\scr{C}_{2^{r}}'$,
\[\xymatrix{
\scr{C}_{2^{r}}'\ar[rr]^{[4]'}\ar[dr]_{j_r'}& & \scr{C}_{2^{r}}'&\\
& \scr{C}_{2^{r-2}}'\ar[ru]_{(4)'}& &,
}\] which induces the commutative diagram
\[\xymatrix{
\mrm{H}^1(G_{F,\Sigma},\scr{C}_{2^{r}}')\ar[rr]^{[4]'_*}\ar[rd]_{(j_r')_*}& &\mrm{H}^1(G_{F,\Sigma}, \scr{C}_{2^{r}}')&\\
&\mrm{H}^1(G_{F,\Sigma}, \scr{C}_{2^{r-2}}')\ar[ru]_{(4)'_*}& .
}\]
Hence, to complete the proof we need to show that $\Sh^1(G_{F,\Sigma},\scr{C}_{2^{r-2}}')$ does not intersect $\ker(4)'_*$ because it would provide the required inclusion $\Sh^1(G_{F,\Sigma},\scr{C}_{2^{r}}')\subset\ker(j_r')_*$.
The exact sequence of $G_{F,\Sigma}$-modules
\[\xymatrix{
1\ar[r]& \scr{C}_{2^{r-2}}'\ar[r]^{(4)'}& \scr{C}_{2^{r}}'\ar[r]& \scr{C}_{2^2}'\ar[r]& 1,
}\]
induces the exact sequence of cohomology groups
\[\xymatrix{
1\ar[r]&C_2=\mrm{H}^0(G_{F,\Sigma}, \scr{C}_{2^2}')\ar[r]^{\delta}& \mrm{H}^1(G_{F,\Sigma}, \scr{C}_{2^{r-2}}')\ar[r]^{(4)_*'}& \mrm{H}^1(G_{F,\Sigma}, \scr{C}_{2^{r}}')
}\]
because any complex conjugation in $G_{F,\Sigma}$ acts by inversion. Hence, $\delta(\mrm{H}^0(G_{F,\Sigma}, \scr{C}_{2^2}'))=\ker(4)_*'$.  Finally, for every real place $v\in\Sigma_\infty$ the connecting homomorphism $\delta_v:C_2=\mrm{H}^0(\bb{R},\scr{C}_{2^2}')\hookrightarrow\mrm{H}^1(\bb{R},\scr{C}_{2^{r-2}}')$ is injective. In particular, the non-trivial class of $\delta(\mrm{H}^0(G_{F,\Sigma}, \scr{C}_{2^2}'))$ is not locally trivial at the real places.
\end{proof}

\begin{lemma}\label{guan}
Let $v$ be a place of $F$, then the local Galois cohomology group $\mrm{H}^2(F_v, \scr{M}_2)$ is trivial.
\end{lemma}
\begin{proof}
If $v$ splits in $M/F$ then $G_{F_v}$ acts trivially on $\scr{M}_2$ and we can refer to Tate's Theorem (Theorem 4, \cite{Serre}). If $v$ is inert or ramified (so non-archimedean under our assumptions), then $G_K$ has cohomological dimension $2$ and $\mrm{H}^2(F_v,\scr{M}_2)$ is $2$-divisible. We conclude by noting that multiplication by $2$ factors through $\mrm{H}^2(M_v,\scr{M}_2)$ which is trivial because $\scr{M}_2$ is a trivial $G_{M_v}$-module.
\end{proof}

\begin{theorem}\label{hook}
Let $F$ be a totally real number field, $\Sigma_0$ a finite set of places of $F$ disjoint from the set $\Sigma_\infty$ of archimedean places and the set $\Sigma_2$ of places above $2$. For $\Sigma$ the complement of $\Sigma_0$, we consider $M/F$ a totally real quadratic extension unramified outside $\Sigma$. Then $\mrm{H}^2(G_{F,\Sigma},\scr{M}_2)=0$.
\end{theorem}
\begin{proof}
By Lemma $\ref{guan}$, it suffices to show that the restriction 
$\mrm{H}^2(G_{F,\Sigma}, \scr{M}_2)\to {\bigoplus}_{v\in\Sigma}\mrm{H}^2(F_v, \scr{M}_2)$ is injective.
For every $r\ge2$, consider the exact sequence 
\[\xymatrix{
0\ar[r]& \Sh^2(G_{F,\Sigma}, \scr{C}_{2^r})\ar[r]& \mrm{H}^2(G_{F,\Sigma}, \scr{C}_{2^r})\ar[r]& \underset{v\in\Sigma}{\bigoplus}\mrm{H}^2(F_v, \scr{C}_{2^r}).
}\]
Poitou-Tate duality (Theorem 8.6.7, \cite{Neu}) gives us a commuting diagram
\[\xymatrix{
\Sh^1(G_{F,\Sigma}, \scr{C}_{2^{r}}')\ar[dd]^{j_{r*}'}& \times& \Sh^2(G_{F,\Sigma}, \scr{C}_{2^{r}})\ar[rrd]& & & \\
&& & & \mathbb{Q}/\mathbb{Z}&\\
\Sh^1(G_{F,\Sigma}, \scr{C}_{2^{r-2}}')&\times& \Sh^2(G_{F,\Sigma}, \scr{C}_{2^{r-2}})\ar[uu]_{j_{r*}}\ar[rru]& & & ,
}\]
which in combination with Proposition $\ref{pio}$, shows that $j_{r*}:\Sh^2(G_{F,\Sigma}, \scr{C}_{2^{r-2}})\to \Sh^2(G_{F,\Sigma}, \scr{C}_{2^{r}})$ is zero because the pairings are perfect. 
Finally, direct limits are exact and commute with direct sums, so 
\[\xymatrix{
0=\underset{r,\to}{\lim}\ \Sh^2(G_{F,\Sigma}, \scr{C}_{2^r})\ar[r]&\mrm{H}^2(G_{F,\Sigma}, \scr{M}_2)\ar[r]& \underset{v\in\Sigma}{\bigoplus}\mrm{H}^2(F_v, \scr{M}_2)
}\] 
is exact.
\end{proof}

\subsubsection{Galois embedding problem.}
Let $n\ge 4$, $r\ge 1$ be integers. The symmetric group $S_n$ acts trivially on $\scr{C}_2$, and it is a classical computation that 
\[
\mrm{H}^2(S_n, \scr{C}_2)\cong \bb{Z}/2\bb{Z}\times \bb{Z}/2\bb{Z}\qquad\mrm{and}\qquad\mrm{H}^2(A_n, \scr{C}_2)\cong \bb{Z}/2\bb{Z}.
\]
We consider a class
\[
[\omega]:\qquad  1\to \scr{C}_2\to \Omega\to S_n\to 1
\]
of $\mrm{H}^2(S_n, \scr{C}_2)$ that does not belong to the kernel of the restriction map $\mrm{H}^2(S_n, \scr{C}_2)\to \mrm{H}^2(A_n, \scr{C}_2)$.

Let $F$ be a totally real field. An $S_n$-Galois extension $J/F$, ramified at a finite set $\Sigma_\mrm{ram}$ of places of $F$, determines a surjection $e:G_{F,\Sigma}\twoheadrightarrow S_n$ where $\Sigma$ is the complement of any finite set $\Sigma_0$ of places of $F$ disjoint from $\Sigma_\mrm{ram}\cup \Sigma_\infty\cup\Sigma_2$. We denote by $M=J^{A_n}$ the fixed field by $A_n$. 

\begin{theorem}\label{ginger}
Suppose the quadratic extension $M/F$ cut out by $A_n$ is totally real. For all $[\omega]\in\mrm{H}^2(S_n,\scr{C}_2)$ restricting to the universal central extension of $A_n$ it is possible to embed $J/F$ into a Galois extension $H/F$ unramified outside $\Sigma$, such that the Galois group $G(H/F)$ represents the non-trivial extension $i_{r*}[\omega]$ of $S_n$ by the $S_n$-module  $\scr{C}_{2^r}$ for some $r\gg0$.
\end{theorem}
\begin{proof}
	Let $i_r: \scr{C}_2\hookrightarrow \scr{C}_{2^r}$ be the natural inclusion. 
The obstruction to the solution of the Galois embedding problem is encoded in the cohomology class $e^*i_{r*}[\omega]\in \mrm{H}^2(G_{F,\Sigma}, \scr{C}_{2^r})$. Indeed, the triviality of the cohomology class is equivalent to the existence of a continuous homomorphism $\gamma:G_{F,\Sigma}\to \Omega_r$ such that the following diagram commutes
\[\xymatrix{
e^*i_{r*}[\omega]:& 1\ar[r]& \scr{C}_{2^r}\ar@{=}[d]\ar[r]& e^*\Omega_r\ar[r]\ar[d]& G_{F,\Sigma}\ar[r]\ar[d]^e\ar@{.>}_\gamma[dl]& 1\\
i_{r*}[\omega]: &1\ar[r]& \scr{C}_{2^r}\ar[r]& \Omega_r\ar[r]& S_n\ar[r]& 1.
}\]
The homomorphism $\gamma$ need not be surjective, but it still defines a non-trivial extension of $S_n$ by a submodule of $\scr{C}_{2^r}$ as $\Omega_r$ is a non-trivial extension. The non-triviality of the class $i_{r*}[\omega]$ follows by the commutativity of the following diagram
\[\xymatrix{
 \mrm{H}^2(S_n, \scr{C}_2)\ar[r]^{i_{r*}}\ar[d]& \mrm{H}^2(S_n, \scr{C}_{2^r})\ar[d]&\\
 \mrm{H}^2(A_n, \scr{C}_2)\ar[r]^{i_{r*}}& \mrm{H}^2(A_n, \scr{C}_{2^r})&
}\]
because by hypothesis the restriction of $[\omega]$ to $\mrm{H}^2(A_n, \scr{C}_2)$ is non-zero and the lower orizontal arrow is injective as $\mrm{H}^1(A_n, \scr{C}_{2^{r-1}})=0$ for $n\ge4$. Finally, the obstruction to the solution of the Galois embedding problem vanishes for $r\gg0$ because $\underset{r,\to}{\lim}\ \mrm{H}^2(G_{F,\Sigma},\scr{C}_{2^r} )=\mrm{H}^2(G_{F,\Sigma},\scr{M}_2)=0$ by Theorem $\ref{hook}$.
\end{proof}

%%%%%%%%%%%%%%%%%%%%%%%%%%%%%%%%%%%%%%%%%%%%%%%%%%%%%%
% On Artin representations
%%%%%%%%%%%%%%%%%%%%%%%%%%%%%%%%%%%%%%%%%%%%%%%%%%%%%%

\section{On Artin representations}
Let $K/F$ be an $S_5$-quintic extension ramified at a finite set $\Sigma_\mrm{ram}$ of places of $F$. Suppose the Galois closure $J$ is totally complex and that the subfield of $J$ fixed by $A_5$ is a totally real quadratic extension $M/F$. Let $\Sigma$ be the complement of a finite set $\Sigma_0$ disjoint from $\Sigma_\mrm{ram}\cup \Sigma_\infty\cup\Sigma_2$.

The simple group $A_5$ does not afford an irreducible $2$-dimensional representation. However, there are two conjugacy classes of embeddings of $A_5$ into $\mrm{PGL}_2(\bb{C})$. We fix one such embedding and we consider the projective representation $G_{M,\Sigma}\twoheadrightarrow G(J/M)\cong A_5\subset \text{PGL}_2(\bb{C})$. We are interested in finding a lift with specific properties.
Consider the double cover $\Omega_5^+$ of $S_5$ where transpositions lift to involutions, and that restricts to the universal central extension of $A_5$ . By Theorem $\ref{ginger}$ there exists a positive integer $r$ and a Galois extension $H/F$, unramified outside $\Sigma$ and containing $J/F$, such that the sequence 
\[\xymatrix{
1\ar[r]& \scr{C}_{2^r}\ar[r]& G(H/F)\ar[r]& G(J/F)\ar[r]& 1
}\]
is exact. Note that transpositions of $ S_5\cong G(J/F)$ lift to element of order 2 of $G(H/F)$ and their conjugation action on $\scr{C}_{2^r}$ take every element to its inverse. Let $\tilde{A}_5$ denote the universal central extension of $A_5$. The representations of the group
\[
G(H/M)\cong (C_{2^{r}}\times \tilde{A}_5)/\langle-1,-1\rangle
\]
are constructed by tensoring a character of $C_{2^r}$ with a $2$-dimensional representation of $\tilde{A}_5$ that takes the same value at $-1$. We consider $\varrho_K: G_{M,\Sigma}\to \mrm{GL}_2(\mathbb{C})$, a representation obtained by composing the quotient map $G_{M,\Sigma}\twoheadrightarrow G(H/M)$ with any irreducible $2$-dimensional representation of $G(H/M)$. 

\begin{remark}\label{dihedral}
Note that since the abelianization of $\tilde{A}_5$ is trivial, there is a dihedral Galois extension $D/F$ such that $\det(\varrho_K)$ factors through the quotient by the subgroup $G(H/D)\cong(C_{2}\times \tilde{A}_5)/\langle-1,-1\rangle$. Therefore, the composition of the determinant with the transfer map, $\det(\varrho_K)\circ V:G_F\longrightarrow\bb{C}^\times$, is the trivial character.
\end{remark}

\begin{proposition}\label{characterization}
The tensor induction $\otimes\mbox{-}\mrm{Ind}_M^F(\varrho_K):G_F\longrightarrow\mrm{GL}_4(\bb{C})$ factors through $G_J$ and induces a faithful representation $\otimes\mbox{-}\mrm{Ind}_M^F(\varrho_K):S_5\longrightarrow\mrm{GL}_4(\bb{C})$ isomorphic to the standard representation of $S_5$ on $5$ letters.
\end{proposition}
\begin{proof}
By construction, the action by conjugation of $G(J/F)$ on $G(H/J)$ factors through $G(M/F)$ and sends every element to its inverse. Let $\theta\in G_F$ be an element mapping to a transposition in $G(J/F)\cong S_5$, then
\[\begin{split}
\ker\left(\otimes\mbox{-}\mrm{Ind}_M^F(\varrho_K)\right)\cap G_M& =\ker\left(\varrho_K\otimes(\varrho_K)^\theta\right)\\
&=\{h\in G_M\lvert\ \exists \alpha\in\bb{C}^\times\ \mrm{with}\ \varrho_K(h)=\alpha\bb{I}_2,\ \varrho_K^\theta(h)=\alpha^{-1}\bb{I}_2\}\\
&=G_{J}.
\end{split}
\]
Thus, $\otimes\mbox{-}\mrm{Ind}_M^F(\varrho_K)$ induces a $4$-dimensional representation $\otimes\mbox{-}\mrm{Ind}_M^F(\varrho_K):S_5\to \mrm{GL}_4(\bb{C})$ of $S_5$.
By Proposition $\ref{tre}$, $\otimes\mbox{-}\mrm{Ind}_M^F(\varrho_K)$ has either decomposition type $(3,1)$ or it is irreducible. Hence, it has to be irreducible since $S_5$ does not have irreducible representations of dimension $3$. Finally, $S_5$ has only two irreducible $4$-dimensional representations: the standard representation $\mrm{St}_{S_5}$ on $5$ letters and its twist by the sign character $\mrm{sign}:S_5\to \{\pm1\}$. We can distinguish between them by computing the trace of transpositions. Recall that our input was the central extension $\Omega_5^+$ of $S_5$ with the property that transpositions of $S_5$ lift to involutions. It follows that  $\theta^2\in G_H$ and $\varrho_K(\theta^2)=\bb{I}_2$, hence we can compute that
\[
\otimes\mbox{-}\mrm{Ind}_M^F(\varrho_K)(\theta)=\begin{pmatrix}
1 &0&0&0\\
0&0&1&0\\
0&1&0&0\\
0&0&0&1\\
\end{pmatrix}
\]
has trace equal to $2$.
\end{proof}

\begin{corollary}\label{exists}
Let $K/F$ be an $S_5$-quintic extension whose Galois closure $J$ is totally complex and contains a totally real quadratic extension $M/F$. Let $\Sigma$ be the complement of any finite set $\Sigma_0$ of places of $F$ disjoint from $\Sigma_\mrm{ram}\cup \Sigma_\infty\cup\Sigma_2$, then there exists a totally odd $2$-dimensional Artin representation $\varrho_K: G_{M,\Sigma}\to \text{GL}_2(\bb{C})$ such that $\otimes\mbox{-}\mrm{Ind}_M^F(\varrho_K)$ is equivalent to $\mrm{Ind}_K^F\bb{I}-\bb{I}$.
\end{corollary}
\begin{proof}
Thanks to Proposition $\ref{characterization}$, we only have to check that the representation $\varrho_K: G_{M,\Sigma}\to \text{GL}_2(\bb{C})$ considered there is totally odd. By assumption the Galois closure $J$ is totally complex, thus the projectivization of $\varrho_K$ is a faithful representation of $G(J/M)$, which contains every complex conjugation of $M$.
\end{proof}

%%%%%%%%%%%%%%%%%%%%%%%%%%%%%%%%%%%%%%%%%%%%%%%%%%%%%%
% Analytic rank of elliptic curves
%%%%%%%%%%%%%%%%%%%%%%%%%%%%%%%%%%%%%%%%%%%%%%%%%%%%%%

\section{Growth of the analytic rank}
Let $M/F$ be a quadratic extension of totally real fields, $E_{/F}$ a modular elliptic curve of conductor $\frak{N}$, and $g$ a primitive Hilbert cuspform over $M$ of parallel weight one and level $\frak{Q}$. Attached to this data, there is a unitary cuspidal automorphic representation $\Pi=\Pi_{g,E}$ of the algebraic group $\mbf{G}=\text{Res}_{M\times F/F}(\text{GL}_{2,M\times F})$. Let $\phi:G_F\to S_3$ be the homomorphism mapping the absolute Galois group of $F$ to the symmetric group over $3$ elements associated with the \'etale cubic algebra $(M\times F)/F$. The $L$-group $^L\mbf{G}$ is given by the semi-direct product $\text{GL}_2(\bb{C})^{\times3}\rtimes_\phi G_F$ where $G_F$ acts on the first factor through $\phi$. 

\begin{definition}
 The twisted triple product $L$-function associated with the unitary autormophic representation $\Pi$ is given by the Euler product
\[
L(s,\Pi,\mrm{r})=\prod_v L_v(s,\Pi_v,\mrm{r})^{-1}
\]
where $\Pi_v$ is the local representation at the place $v$ of $F$ appearing in the restricted tensor product decomposition $\Pi=\bigotimes_v'\Pi_v$, and the representation $\mrm{r}$ gives the action of $^L\mbf{G}$ on $\bb{C}^2\otimes\bb{C}^2\otimes\bb{C}^2$ which restricts to the natural $8$-dimensional representation of $\text{GL}_2(\bb{C})^{\times3}$ and for which $G_F$ acts via $\phi$ permuting the vectors.
\end{definition}
Assume the central character $\omega_\Pi$ of $\Pi$ is trivial when restricted to $\bb{A}_F^\times$, then the complex $L$-function $L(s,\Pi,\mrm{r})$ has meromorphic continuation to $\bb{C}$ with possible poles at $0,\frac{1}{4},\frac{3}{4},1$ and functional equation $L(s,\Pi,\mrm{r})=\epsilon(s,\Pi,\mrm{r})L(1-s,\Pi,\mrm{r})$ (\cite{PS-R}, Theorems 5.1, 5.2, 5.3). When all the primes dividing $\frak{N}$ are unramified in $M/F$ and $(\frak{N}, \mrm{N}_{M/F}(\frak{Q}))=1$, the sign of the functional equation can be computed as follows (Theorems $\mrm{B},\mrm{D}$, Remark 4.1.1, \cite{epsilonprasad}). Write $\frak{N}=\frak{N}^+\frak{N}^-$, where $\frak{N}^-$ is the square-free part of $\frak{N}$, and suppose that all prime factors of $\frak{N}^+$ are split in $M/F$, then the sign of the functional equation is determined by the number of prime divisors of $\frak{N}^-$ which are inert in $M/F$:
\[
\epsilon\Big(\frac{1}{2},\Pi,\mrm{r}\Big)=\left(\frac{M/F}{\frak{N}^-}\right).
\]

\begin{theorem}\label{parity}
Let $E_{/F}$ be a modular elliptic curve of odd conductor $\frak{N}$ and let $K/F$ be an $S_5$-quintic extension with totally complex Galois closure $J$. Suppose $J$ is  unramified at $\frak{N}$ and contains a totally real quadratic extension $M/F$, then the ratio of $L$-functions $L(E/K,s)/L(E/F,s)$ has meromorphic continuation to the whole complex plane and it is holomorphic at $s=1$. Furthermore, if all prime factors of $\frak{N}^+$ are split in $M/F$, then 
	\[
	 \mrm{ord}_{s=1}\frac{L(E/K,s)}{L(E/F,s)}\equiv1 \pmod{2}\qquad \iff\qquad\left(\frac{M/F}{\frak{N}^-}\right)=-1.
	 \]
\end{theorem}
\begin{proof}
  Thanks to Corollary $\ref{exists}$ and the modularity of totally odd Artin representations of the absolute Galois group of totally real fields (Theorem 0.3, \cite{PS}), there is a primitive Hilbert cuspform $g$ of parallel weight one over $M$, and level $\frak{Q}$ prime to $\frak{N}$, such that $\varrho_g=\varrho_K$.  A direct inspection of the Euler product of the twisted triple product $L$-function $L(s,\Pi,\mrm{r})$ attached to $\Pi=\Pi_{g,E}$ produces the equality of incomplete $L$-functions
  \[
  L_S(s,\Pi,\mrm{r})=L_S\Big(E,\otimes\mbox{-}\mrm{Ind}_M^F(\varrho_g),s+\frac{1}{2}\Big)=\frac{L_S(E/K,s+\frac{1}{2})}{L_S(E/F,s+\frac{1}{2})},
  \]
   where $S$ is any finite set containing the primes dividing $\frak{N}\cdot\mrm{N}_{M/F}(\frak{Q})$ and the primes that ramify in $M/F$. Remark $\ref{dihedral}$ ensures the triviality of the central character $\omega_\Pi$ when restricted to $\bb{A}^\times_F$, hence, meromorphic continuation, holomorphicity at the center and the criterion for the parity of the order of vanishing at the center of $L(E/K,s)/L(E/F,s)$ follow.
\end{proof}

\begin{corollary}
	Let $E_{/F}$ be an elliptic curve of odd conductor $\frak{N}$ and at least one prime of multiplicative reduction. We denote by $G_5(E_{/F};X)$ the number of quintic extensions $K$ of $F$ such that the norm of the relative discriminant is at most $X$ and the analytic rank of $E$ grows over $K$, i.e., $r_\mrm{an}(E/K)>r_\mrm{an}(E/F)$. Then 	
$G_5(E;X)\asymp_{+\infty} X$.
\end{corollary}
\begin{proof}
	By Theorem $\ref{parity}$, $G_5(E_{/F};X)$ contains all $S_5$-quintic extension $K/F$  with totally complex Galois closure $J$ containing a totally real quadratic extension  in which the prime divisors of $\frak{N}$ are unramified and have certain splitting behaviour. Then (\cite{BSW}, Theorem $2$) gives $G_5(E_{/F};X)\gg_{+\infty} X$, while (\cite{BSW}, Theorem $1$) provides $X\gg_{+\infty} G_5(E_{/F};X)$.
\end{proof}

%%%%%%%%%%%%%%%%%%%%%%%%%%%%%%%%%%%%%%%%%%%%%%%%%%%%%%
% BIBLIOGRAPHY
%%%%%%%%%%%%%%%%%%%%%%%%%%%%%%%%%%%%%%%%%%%%%%%%%%%%%%

\bibliography{p3L_GZ}
\bibliographystyle{alpha}

\end{document}